\documentclass{amsart}

\usepackage[letterpaper,top=2cm,bottom=2cm,left=3cm,right=3cm,marginparwidth=1.75cm]{geometry}

\usepackage{amsmath, amssymb}
\usepackage{graphicx}
\usepackage{xcolor}
\usepackage[colorlinks=true, allcolors=blue]{hyperref}

\numberwithin{equation}{section}

\newtheorem{theorem}{Theorem}[section]

\newtheorem{addendum}[theorem]{Addendum}
\newtheorem{corollary}[theorem]{Corollary}
\newtheorem{prop}[theorem]{Proposition}
\newtheorem{conj}[theorem]{Conjecture}
\newtheorem{lemma}[theorem]{{\bf Lemma}}

\newcounter{hyp_counter}
\theoremstyle{definition}

\theoremstyle{remark}
\newtheorem{remark}[theorem]{Remark}

\newcommand{\R}{\mathbb R}
\newcommand{\Z}{\mathbb Z}
\newcommand{\T}{\mathbb T}

\newcommand{\U}{\mathcal U}
\newcommand{\cH}{\mathcal H}
\newcommand{\eps}{\varepsilon}

\renewcommand{\phi}{\varphi}

\newenvironment{psmallmatrix}
  {\left(\begin{smallmatrix}}
  {\end{smallmatrix}\right)}

\title{Symplecto-rigidity and bootstrap for Anosov symplectomorphisms and contact Anosov flows}
\author{Andrey Gogolev}
\address{Department of Mathematics, The Ohio State University, Columbus, OH 43210, USA}
\email{gogolyev.1@osu.edu}
\author{Federico Rodriguez Hertz}
\address{Department of Mathematics, The Pennsylvania State University, University Park, PA 16802, USA}
\email{hertz@math.psu.edu}
\thanks{The authors were partially supported by NSF grants DMS-2247747 and DMS-2453688, respectively}

\begin{document}

\begin{abstract}
This paper is a sequel to the authors' paper on rigidity of higher dimensional contact Anosov flows~\cite{GRH}. At the time the authors were unaware of much earlier work of Hamenst\"adt~\cite{Ham} devoted to the same problem. The methods of~\cite{GRH} and~\cite{Ham} are overlapping but not entirely the same. In this paper we strengthen some the results of Hamenst\"adt, by further bootstrapping the regularity of the conjugacy of 2-pinched contact Anosov flows. We also introduce a notion of symplecto-rigidity for symplectomorphisms and prove that some Anosov diffeomorphisms are symplecto-rigid. We further combine symplecto-rigidity with various rigidity techniques to establish a number of smooth rigidity results of Anosov symplectomorphisms. Some new phenomena are present in the realm of Anosov symplectomorphism. For example, while the well-known de la Llave example on the 4-torus demonstrates absence of rigidity in the space of smooth Anosov diffeomorphisms, however, we prove that it is rigid in the space of Anosov sympelctomorphisms.
\end{abstract}

\maketitle

\section{Introduction and new results}

\subsection{Symplecto-rigidity}
Let $M$ be a closed manifold equipped with a symplectic form $\omega$. A symplectomorphism $f\colon (M,\omega)\to (M,\omega)$ is {\it symplecto-rigid} if $\omega$ is the unique continuous invariant symplectic form in its cohomology class; equivalently, $f$ admits no non-trivial invariant exact $2$-form. Indeed, the difference of two cohomologous invariant symplectic forms is an invariant exact $2$-form. Conversely, if $\eta$ is a non-zero continuous invariant exact $2$-form, then $\omega+t\eta$ is closed, invariant, and cohomologous to $\omega$; since non-degeneracy is open in the $C^0$ topology, it is symplectic for all sufficiently small $t$. While the definition makes sense beyond hyperbolic dynamics, here we focus on aspects of symplecto-rigidity pertaining to Anosov symplectomorphisms, since this will be relevant to bootstrapping regularity of conjugacies between Anosov symplectomorphisms. In fact, for such purposes, we could use a weaker version of the definition. Given an $\alpha\ge0$, a symplectomorphism $f$ is called {\it $C^\alpha$-symplecto-rigid} if $\omega$ is the unique invariant $C^\alpha$-regular symplectic form in its cohomology class. 

\begin{conj} Every Anosov symplectomorphism is symplecto-rigid.
\end{conj}
Yong Fang proved this conjecture for Anosov automorphisms of nilmanifolds~\cite{Fang}. His theorem is stated for smooth forms, but the proof also works for continuous forms. For completeness, we include a straightforward Fourier proof for toral automorphisms in Appendix~\ref{app_toral_forms}, which is the case used here.

\begin{theorem}
\label{thm_dim4}
Let $L\colon(\T^4,\omega) \to(\T^4,\omega)$ be a symplectic Anosov automorphism. There exists a small $\alpha>0$ such that any Anosov symplectomorphism $f\colon (\T^4,\omega)\to(\T^4,\omega)$, which is sufficiently $C^1$ close to $L$, is $C^\alpha$-symplecto-rigid.
\end{theorem}

\begin{addendum}
\label{add_dim4}
Using work of Kalinin and Sadovskaya and an additional
measurable Livshits argument, the conclusion of
Theorem~\ref{thm_dim4} can be strengthened: under the same assumptions $f$
is symplecto-rigid.
\end{addendum}

\begin{theorem}
\label{theorem_residual}
A residual set of Anosov symplectomorphisms is symplecto-rigid.
\end{theorem}
{Here by a residual set we mean a countable intersection of subsets that are open and dense in the $C^1$ topology on the space of Anosov symplectomorphisms of $(M,\omega)$.}

\begin{remark} The notion of symplecto-rigidity together with the stronger property of {\it symplecto-ergodicity} will be developed in a companion paper~\cite{GRH27}. 
If $f\colon (M,\omega)\to (M,\omega)$ is a symplectomorphism then the non-vanishing top-dimensional form $vol=\omega^{\wedge d}$ is invariant. Generally speaking, symplecto-ergodicity and ergodicity with respect to $vol$ do not imply one another~\cite{GRH27}. 
\end{remark}

\subsection{Bootstrap for Anosov symplectomorphisms}

Let $(M,\omega)$ be a closed symplectic manifold and let $f\colon (M,\omega) \to (M,\omega)$ be an Anosov symplectomorphism with invariant splitting $TM=E^s\oplus E^u$. Throughout we will assume that Anosov diffeomorphisms we consider are $C^\infty$ smooth, however all results also hold in appropriate finite regularity as well.

We call $f$ {\it conformally $r$-pinched} (or just {\it $r$-pinched}) if for a sufficiently large $n$ and all $x\in M$
\begin{equation}
\label{eq_pinch_symp}
\|Df^n|_{E^u(x)}\|<m(Df^n|_{E^u(x)})^r,
\end{equation}
where $m$ denotes the conorm, $m(A)=\|A^{-1}\|^{-1}$. 

\begin{theorem}
\label{thm_symp}
Let $f_1, f_2\colon (M,\omega)\to (M,\omega)$ be symplecto-rigid
Anosov symplectomorphisms whose stable and unstable distributions are
of class $C^{1+\eps}$ for some $\eps\in(0,1)$. Assume that $f_1$ and
$f_2$ are conjugate via a $C^1$ diffeomorphism which is homotopic to
the identity. Then the conjugacy is, in fact, a $C^{2+\eps}$ diffeomorphism.
\end{theorem}
Using uniqueness of normal forms, we further obtain the following.
\begin{corollary}
\label{cor_symp}
    Let $f_1, f_2\colon (M,\omega)\to (M,\omega)$ be symplecto-rigid Anosov symplectomorphisms which are $2$-pinched. Assume that $f_1$ and $f_2$ are conjugate via a $C^1$ diffeomorphism homotopic to the identity. Then the conjugacy is, in fact, a $C^{\infty}$ diffeomorphism.
\end{corollary}

We can also give some applications to periodic data rigidity of Anosov diffeomorphisms. Consider two Anosov diffeomorphisms $f_1$ and $f_2$ which are conjugate via a homeomorphism $h$, $h\circ f_1=f_2\circ h$. Recall that $f_1$ and $f_2$
have the same {\it periodic data} with respect to $h$, if for each periodic point
$p=f_1^k(p)$ the differentials $D_pf_1^k$ and $D_{h(p)}f_2^k$ are conjugate. 
Also recall that an Anosov toral automorphism $L$ is {\it irreducible} if the characteristic polynomial of $L$ is irreducible over $\Z$.

Combining the above result with symplecto-rigidity of automorphisms~\cite{Fang} and work of Gogolev~\cite{Gog} and DeWitt-Gogolev~\cite{DWG} we have the following corollary.
\begin{corollary}
\label{cor_periodic_data}
Let $L\colon(\T^d,\omega)\to(\T^d,\omega)$ be an irreducible symplectic $2$-pinched hyperbolic automorphism such that no three of its
eigenvalues have the same modulus. Assume that $f\colon (\T^d,\omega)\to(\T^d,\omega)$ is a $C^\infty$ Anosov symplectomorphism conjugate to $L$ via a continuous conjugacy $h$ with the same periodic data. Then $h$ is a $C^\infty$ diffeomorphism.
\end{corollary}

We note that this is now a particular case of a recent spectacular result of Kalinin--Sadovskaya--Wang~\cite{KSW}, who do not require symplecticity or pinching. However, it is still of interest, as our bootstrap argument is different, shorter, and more elementary.

The next corollary also follows directly from  Theorem~\ref{thm_dim4} and Corollary~\ref{cor_symp} in conjunction with rigidity results in~\cite{Gog, DWG}.
\begin{corollary}
\label{cor_real_spectrum}
Let $L\colon(\T^4,\omega)\to(\T^4,\omega)$ be an irreducible symplectic $2$-pinched hyperbolic automorphism with real spectrum. Then there exists a $C^1$ neighborhood $\U\ni L$ such that if $f\in\U$ is a symplectomorphism of $(\T^4,\omega)$ and an Anosov symplectomorphism $g\colon(\T^4,\omega)\to(\T^4,\omega)$ is continuously conjugate to $f$ with the same periodic data, then the conjugacy is a $C^\infty$ diffeomorphism.
\end{corollary}

We finish this discussion with an interesting twist: it turns out that symplecticity can sometimes replace irreducibility in smooth rigidity problems.

First we recall the de la Llave example~\cite{L}. Let $A$ and $B$ be automorphisms of the 2-torus $\T^2$ induced by hyperbolic matrices in $SL(2,\Z)$. We will assume that the smaller eigenvalues $\lambda$ and $\mu$ of $A$ and $B$, respectively, satisfy the following pinching inequalities 
\begin{equation}
\label{eq_pinch}
0<\mu^2<\lambda<\mu<1.
\end{equation}
 Consider the linear automorphism $L=A\times B$. De la Llave considered the perturbations of the form
$$
L_{\phi}(x,y)=(Ax, By+\phi(x)), \,\,\,\,(x,y)\in\T^2\times\T^2
$$
and proved that, in general, the conjugacy between $L$ and $L_\phi$ is merely $C^{\log\mu/\log\lambda}$ despite the fact that they have the same periodic data. Note, however, that the de la Llave perturbation $L_\phi$ is not symplectic with respect to the standard symplectic form $\omega_0=dx_1\wedge dx_2+dy_1\wedge dy_2$.

\begin{theorem}
\label{cor_product}
Let $L\colon(\T^4,\omega_0)\to(\T^4,\omega_0)$ be a $2$-pinched product hyperbolic automorphism as described above. Then any sufficiently $C^1$-small symplectic perturbation $f$ with the same periodic data is $C^\infty$ conjugate to $L$.
\end{theorem}

\begin{remark}
    In the case when $A=B$ (and $\lambda=\mu$) de la Llave examples do not have the same periodic data as the automorphism. This case was considered by de la Llave who established $C^\infty$ periodic data rigidity~\cite{L2} without any additional assumptions. The case of 4-dimensional automorphisms with Jordan blocks was studied by DeWitt~\cite{DW25} who proved $C^{1+}$ regularity of the conjugacy when eigenvalue and Jordan periodic data coincide for the perturbation. Under the additional assumption that the perturbation is symplectic (preserves a smooth symplectic form), Corollary~\ref{cor_symp} applies to further bootstrap regularity of conjugacy to $C^\infty$. 
\end{remark}

\subsection{Bootstrap for contact Anosov flows} 
Four years ago, we wrote the paper~\cite{GRH}, whose main result was that $C^0$-conjugate contact flows with $C^r$, $r>1$, stable and unstable distributions are, in fact, $C^{r-\eps}$-conjugate for any $\eps>0$. When writing~\cite{GRH}, we were regrettably unaware of the much earlier work of Hamenst\"adt~\cite{Ham} devoted to the same question. We are very grateful to Karen Butt, who pointed out this reference to us. The two papers overlap significantly in their results; however, their proof techniques are quite different.

Hamenst\"adt concludes $C^2$ regularity of the conjugacy for flows with $C^{1}$ distributions by first establishing Lipschitz regularity, then $C^1$ regularity, and finally bootstrapping to $C^2$ with the help of the Kanai connection. The present authors used the matching functions technique together with non-stationary normal forms~\cite{GRH}, and could bootstrap the regularity from $C^1$ to $C^\infty$ only when the distributions were assumed to be at least $C^{\sqrt{2}}$.
In this follow-up paper, we first give a self-contained exposition of Hamenst\"adt's bootstrap argument~\cite{Ham} and then combine it with our techniques from~\cite{GRH} to bootstrap the regularity of the conjugacy in the setting of contact flows.

Let $M$ be a contact manifold equipped with a contact 1-form $\alpha$. Let $\phi^t\colon M\to M$ be an Anosov flow with invariant splitting $TM=E^s\oplus \R X\oplus E^u$, $X=\partial\phi^t/\partial t$, which preserves the contact structure: $(\phi^{t})^*\alpha=\alpha$. We always normalize $\alpha$ so that $\alpha(X)=1$. We call such a flow {\it conformally $r$-pinched} (or simply $r$-pinched) if for a sufficiently large $t$ and all $x\in M$
\begin{equation}
\label{eq_pinch_cont}
\|D\phi^t|_{E^u(x)}\|<m(D\phi^t|_{E^u(x)})^r,
\end{equation}

\begin{theorem}[Hamenst\"adt~\cite{Ham}]
\label{thm_contact}
For $i=1,2$, let $(M_i,\alpha_i)$ be a closed contact manifold and let
$\phi_i^t\colon M_i\to M_i$ be a contact Anosov flow whose stable and
unstable distributions are of class $C^{1+\eps}$ for some
$\eps\in(0,1)$. If $\phi_1^t$ and $\phi_2^t$ are conjugate by a
time-preserving homeomorphism, then the conjugacy is a
$C^{2+\eps}$ diffeomorphism.
\end{theorem}
Utilizing uniqueness of normal forms we have.
\begin{corollary}
\label{cor_contact}
For $i=1,2$, let $\phi_i^t\colon M_i\to M_i$ be a contact Anosov flow
which is $2$-pinched. If $\phi_1^t$ and $\phi_2^t$ are conjugate by a
time-preserving homeomorphism, then the conjugacy is a $C^{\infty}$
diffeomorphism.
\end{corollary}
In particular, this corollary applies to geodesic flows of metrics with $\frac14$-pinched negative sectional curvature (cf.~\cite[Corollary~1.4]{GRH}).

\subsection*{Organization}
Section~2 reviews bunching, non-stationary linearization, and Kanai parallelism. Section~3 proves the bootstrap results for contact Anosov flows, and Section~4 proves the smooth-conjugacy results for Anosov symplectomorphisms. Section~5 establishes the symplecto-rigidity results. The appendices contain the Fourier argument for invariant forms on tori and the regularity calculation for holonomy-invariant vector fields.

\subsection*{Acknowledgments}
Most of all we are very thankful to Karen Butt for a discussion at Oberwolfach in July 2025 and for pointing out the reference~\cite{Ham}. We also would like to thank Boris Kalinin for discussions on rigidity.
\subsection*{AI} The authors completed the bulk of the paper in the summer of 2025 without AI assistance. While finishing the proofs of symplecto-rigidity in the summer of 2026, conversations with Claude Opus 4.8 and ChatGPT 5.6 were helpful. 
The paper was written by the authors, except that initial drafts of the proofs of Theorem~\ref{theorem_residual} and Appendix~A were generated by ChatGPT and subsequently edited by the authors. We also used Claude Opus 4.8 and ChatGPT 5.6 to check and proofread the paper. Needless to say, the authors bear full responsibility for its correctness.

\section{Preliminaries}

\subsection{Bunching and regularity of distributions}
\label{sec:bunching}
We fix a background Riemannian metric $\langle \cdot,\cdot\rangle$ on $M$, which induces the norm $\|\cdot \|$. We will recall the well-known bunching and pinching conditions for Anosov diffeomorphisms and their implications for the regularity of invariant distributions. For Anosov flows, the parallel discussion of regularity is more subtle because strong stable subbundles are not as regular as weak subbundles. However, in the contact case, the strong subbundles are the intersections of the weak subbundles with the kernel of the contact form, which is smooth. Hence the strong subbundles have exactly the same regularity as the weak ones, and the discussion is fully analogous to the diffeomorphism case.

The stable distribution $E^s$ of an Anosov diffeomorphism $f$ satisfies the {\it $r$-bunching condition} if for some $n>0$ and all $x$
$$
\|Df^n|_{E^s(x)}\|\cdot \|Df^n|_{E^u(x)}\|^{r}<m(Df^n|_{E^u(x)})
.$$
This condition implies $C^r$ regularity of $E^s$~\cite{Hass}. Analogously, the condition
$$
\|Df^n|_{E^s(x)}\|<m(Df^n|_{E^u(x)})\cdot  m(Df^n|_{E^s(x)})^{r}
$$
implies that the unstable distribution $E^{u}$ is $C^{r}$.

For Anosov symplectomorphisms, the symmetry between the stable and unstable subbundles simplifies the bunching conditions. Indeed, if $v\in E^s(x)$ is non-zero, then $\omega(v,\cdot)\colon E^u(x)\to\R$ is a nonzero linear functional and hence is represented by the scalar product as $\omega(v,\cdot)=\langle \cdot,\hat v\rangle$ for some $\hat v\in E^u(x)$. Therefore, we have $m(Df^n|_{E^s(x)})\simeq\|Df^n|_{E^u(x)}\|^{-1}$ and $\|Df^n|_{E^s(x)}\|\simeq m(Df^n|_{E^u(x)})^{-1}$ (up to a uniform constant). This symmetry immediately implies that the two $r$-bunching conditions above are equivalent and reduce to
$$
\|Df^n|_{E^u(x)}\|<m(Df^n|_{E^u(x)})^{\frac2r}.
$$
Consequently, the $r$-pinching assumption~\eqref{eq_pinch_symp} implies that both distributions $E^s$ and $E^u$ satisfy $\frac2r$-bunching and hence are $C^{2/r}$-regular.

\subsection{Pinching and uniqueness of non-stationary linearization}
We will rely a lot on Guysinsky-Katok non-stationary normal form theory~\cite{GK}. We the following normal form result for 2-pinched diffeomorphisms is due to Sadovskaya.
\begin{prop}[\cite{GK,Sad05}]
\label{prop_normal_forms}
	Let $f\colon M\to M$ be an Anosov diffeomorphism. Assume that $f$ is 2-pinched, that is, there exists $n>0$ such that for all $x\in M$
	$$
    \|Df^n|_{E^u(x)}\|<m(Df^n|_{E^u(x)})^{2}.
 $$
	Then for all $x\in M$ there exists $\cH_x:E^u(x)\to W^u(x)$ such that
	\begin{enumerate}
		\item $\cH_x$ is a smooth diffeomorphism for all $x\in M$;
		\item $\cH_x(0)=x$;
		\item $D_0\cH_x=id$;
		\item $\cH_{f^nx}\circ Df^n|_{E^u(x)} =f^n\circ \cH_x$ for all $n$;
        \item Non-stationary linearization $\{\cH_x, x\in M\}$ is unique among $C^2$ smooth linearizations satisfying properties $(2)-(4)$ above.
 	\end{enumerate}
\end{prop}

For more details, especially discussion of uniqueness, we also refer to~\cite{Kal} and to~\cite{GRH} for a proof of an even more precise uniqueness statement.

\subsection{Kanai parallelism}
\label{sec:kanai}

Kanai constructed a connection adapted to an Anosov symplectomorphism with $C^r$, $r\ge 1$, invariant distributions~\cite{K1,K2}. We recall the idea and provide some details which will be important for us. Let $\gamma\colon[0,1]\to W^u(x)$ be a smooth curve connecting $a=\gamma(0)$ to $b=\gamma(1)$. Given a vector $V_0\in E^u(a)$ Kanai defines its parallel transport $\{V_t\}$ along $\gamma$ in the following way. Let $H^u_t\colon W^s_{loc}(a)\to W^s_{loc}(\gamma(t))$ be the family of unstable holonomies along $\gamma$. By our assumption these holonomies are $C^r$ and given a vector $W_0\in E^s(a)$ we obtain a vector field $W_t=DH^u_t(W_0)\in E^s(\gamma(t))$.
\begin{lemma} \label{lemma}
The holonomy invariant vector field $W_t$ is $C^{r}$.
\end{lemma}
We give a sketch of the proof of this calculus lemma in the appendix.

Now a vector $V_0\in E^u(a)$ defines a linear functional $\hat V_0\colon E^s(a)\to\R$ given by $\hat V_0(\cdot)=\omega(V_0,\cdot)$. We extend it along $\gamma$ by
$$
\hat V_t=\hat V_0\circ(D_aH^u_t)^{-1}.
$$
Again, since $\omega$ is non-degenerate, there exists a unique $V_t\in E^u(\gamma(t))$ such that $\hat V_t=\omega(V_t,\cdot)$.

\begin{lemma} \label{lemma2} The dual field $\hat V_t$ and the vector field $V_t$ are $C^{r}$ regular.
\end{lemma}
\begin{proof}
Fix a scalar product $\langle\cdot,\cdot\rangle_0$ on $E^s(a)$ and push it forward to a scalar product $\langle\cdot,\cdot\rangle_t$ on $E^s(\gamma(t))$, $t\in[0,1]$, via the unstable holonomy $D_pH^u_t$. Pushing forward an orthonormal basis of $E^s(a)$ yields an orthonormal basis of $E^s(\gamma(t))$. By Lemma~\ref{lemma}, the basis vectors are $C^{r}$ with respect to $t$. It follows that $\langle\cdot,\cdot\rangle_t$ is also $C^{r}$ in $t$.

There is a unique $W_0$ such that $\hat V_0=\langle W_0,\cdot\rangle_0$; set $W_t=D_pH^u_t(W_0)$. Then
$$
\hat V_t=\langle W_t,\cdot\rangle_t
.$$
Since both $W_t$ and $\langle\cdot,\cdot\rangle_t$ are $C^{r}$ in $t$, so is $\hat V_t$. Finally, since $\omega$ is smooth we can also conclude that $V_t$ is $C^{r}$.
\end{proof}

Hence we have constructed a $C^{r}$ parallel vector field $V_t$ along $\gamma$. In particular, we have transported $V_0\in E^u(a)$ to $V_1\in E^u(b)$. Note that $V_1$ is independent of the choice of $\gamma$ and hence is canonically defined. In other words, the connection providing the parallel transport is flat. Indeed, the holonomy $DH^u_1\colon E^s(a)\to E^s(b)$ used to define $V_1$ depends only on the homotopy class of $\gamma\subset W^u(a)$, and $W^u(a)$ is simply connected.

\begin{remark} Analogously the symplectic form $\omega$ and the stable holonomies give rise to parallel transport along stable leaves. All this information can be incorporated into a single {\it Kanai connection} which, in fact, has further geometric meaning.
Namely, the {\it Avez metric} is defined by
$$
g(v,w)=\omega(v, Iw)
$$
where $I(w^u+w^s)=w^u-w^s$, $w=w^u+w^s\in E^u\oplus E^s$. It is immediate to check that $g$ is a non-degenerate pseudo-Riemannian metric of signature $(n,n)$ and that the stable and unstable subspaces are Lagrangian.
Then the  Kanai connection is the unique affine connection with respect to which $g$ is parallel and whose torsion vanishes~\cite{K2}.
\end{remark}

\subsection{Non-stationary linearization via Kanai parallelism}

Pick a basis $\{e_1^a, e_2^a, \ldots e_n^a\}$ of $E^u(a)$. Then, via the parallel transport, we have a basis $\{e_1^b, e_2^b, \ldots e_n^b\}$ of $E^u(b)$ for every $b\in W^u(a)$. This family of frames constitutes a {\it parallel framing} of the unstable manifold. Each pair of vector fields $e_i$ and $e_j$ commute. 


We next explain how the parallel framing gives another construction of the non-stationary linearization. The link between the Kanai connection and normal forms is natural here; however, our proofs will not rely on it directly. For each $i$, the equation $\dot\psi_i=e_i$ defines a flow $\psi_i^s$ on $W^u(a)$. These flows commute, so we can define $\cH_a\colon E^u(a)\to W^u(a)$ by setting
$$
\cH_a\left(\sum_i s_ie_i^a\right)=\psi_1^{s_1}\circ \psi_2^{s_2}\circ\ldots\circ \psi_n^{s_n}(a)
.$$
Differentiating with respect to $s_1$ gives $D\cH_a(\frac{\partial}{\partial s_1})=e_1$. Since the order in the above composition does not matter, we can rearrange the factors so that $\psi_i^{s_i}$ is last and similarly obtain $D\cH_a(\frac{\partial}{\partial s_i})=e_i$. Recall that the vector fields $e_i$ are $C^{r}$ by Lemma~\ref{lemma2}. Hence $\cH_a$ is $C^{r+1}$.

In this way we obtain a family of parametrizations $\{\cH_a:a\in M\}$. We have $\cH_a(0)=a$ and $D\cH_a(e_i^a)=e_i^a$. A straightforward diagram chase, using the fact that the dynamics commutes with holonomy, shows that the differential of $\cH_{f^na}^{-1}\circ f^n\circ \cH_a$ is a constant linear map. Hence $\cH_{f^na}^{-1}\circ f^n\circ \cH_a$ itself is linear. Therefore, since $r+1\ge2$, item~$(5)$ of Proposition~\ref{prop_normal_forms} shows that we have recovered the non-stationary linearization from Kanai parallelism.

\subsection{Kanai parallelism for contact flows}
\label{sec:contact-kanai}
Let $\phi^t$ be a contact Anosov flow preserving a contact form $\alpha$.
By invariance, the form $d\alpha$ vanishes on $E^s$ and on $E^u$,
and pairs $E^s$ and $E^u$ non-degenerately. Consequently, the construction
in Subsection~\ref{sec:kanai} applies verbatim, with $d\alpha$ in place of
the symplectic form: unstable holonomy defines parallel transport along
strong unstable leaves, and stable holonomy defines parallel transport
along strong stable leaves. These are the leafwise parallelisms induced
by the Kanai connection~\cite{K1,K2,Ham}. 

\section{Proofs --- Hamenst\"adt's theorem}
Here we give an exposition of Hamenst\"adt's theorem
(Theorem~\ref{thm_contact}) and its Corollary~\ref{cor_contact}.

\subsection{Step 1: the conjugacy is a $C^{1+\eps}$ diffeomorphism}
This step uses the joint non-integrability of the stable and unstable
bundles of a contact flow and the fact that the conjugacy $h$ preserves
the temporal function, which is $C^{1+\eps}$. The inverse function
theorem then gives $C^{1+\eps}$ regularity of both $h$ and $h^{-1}$.
This argument was first discovered by Feldman and Ornstein in dimension
three~\cite{FO} and then improved by Hamenst\"adt~\cite{Ham} and the
current authors~\cite{GRH}. Hamenst\"adt's proof is different, even
though it exploits the same basic mechanism of non-integrability: she
first establishes Lipschitz continuity and then improves it to $C^1$.

\subsection{Step 2: a $C^{1+\eps}$ conjugacy is $C^{2+\eps}$}
By Step~1, the conjugacy $h$ is $C^{1+\eps}$. Since
$Dh(X_1)=X_2$ and $Dh(\ker\alpha_1)=\ker\alpha_2$, we have
$h^*\alpha_2=\alpha_1$ and hence $h^*d\alpha_2=d\alpha_1$.

Fix $a$ and a basis $\{e_1^a,\ldots,e_n^a\}$ of $E_1^u(a)$, and extend
it to a parallel frame $\{e_1,\ldots,e_n\}$ on $W_1^u(a)$. Push the
basis at $a$ forward by $Dh_a$ and extend it to a parallel frame
$\{\bar e_1,\ldots,\bar e_n\}$ on $W_2^u(h(a))$. The contact Kanai
parallelism from Subsection~\ref{sec:contact-kanai} depends only on the
holonomies and on $d\alpha_i$. Because $h$ conjugates the holonomies and
$h^*d\alpha_2=d\alpha_1$, it follows that
$$
Dh(e_i)=\bar e_i,\qquad i=1,\ldots,n.
$$
Recall that parallel framing is $C^{1+\eps}$ regular. Hence, with
respect to a smooth coordinate chart, the above equation becomes a
system of linear equations on partial derivatives
$\partial h_j/\partial x_i$ along the
unstable leaves with $C^{1+\eps}$ coefficients. More explicitly, if
$E$ and $\bar E$ denote the frame matrices in these coordinates, then
$$
Dh\cdot E=\bar E\circ h,
\qquad\text{and hence}\qquad
Dh=(\bar E\circ h)E^{-1}.
$$
Since $E$ and $\bar E$ are $C^{1+\eps}$ and $h$ is $C^1$, the right
hand side is $C^1$. Hence $h$ is $C^2$ along unstable leaves. Now
differentiate the last formula. The derivative of $Dh$ is a sum of
terms involving $D\bar E\circ h$, $Dh$, $E^{-1}$ and $D(E^{-1})$.
At this point, along each unstable leaf, $h$ and $Dh$ are Lipschitz,
while $D\bar E$ and
$D(E^{-1})$ are $C^\eps$. Therefore $D(Dh)$ is $C^\eps$. Thus $h$ is
$C^{2+\eps}$ along the unstable leaf $W_1^u(a)$. Since the choice of $a$
was arbitrary, we obtain that $h$ is $C^{2+\eps}$ along the unstable
leaves, in fact, uniformly so, since all the considerations were
uniform with respect to the choice of $a$. Then the symmetric argument
gives $C^{2+\eps}$ smoothness along the strong stable foliation.

Finally, $h$ is smooth along flow orbits because the conjugacy is
time-preserving. Journ\'e's lemma~\cite{J}, applied first inside each
weak stable leaf to the strong stable and flow foliations, gives
$C^{2+\eps}$ regularity along weak stable leaves. Applying it once more
to the weak stable and strong unstable foliations gives global
$C^{2+\eps}$ regularity. This completes the proof of
Theorem~\ref{thm_contact}.

\begin{proof}[Proof of Corollary~\ref{cor_contact}]
By compactness, $2$-pinching implies $r$-pinching for some $r<2$.
The discussion in Subsection~\ref{sec:bunching} then gives
$C^{1+\eps}$ stable and unstable distributions for some $\eps>0$.
Theorem~\ref{thm_contact} makes the conjugacy $C^{2+\eps}$. Uniqueness
of the non-stationary linearizations in
Proposition~\ref{prop_normal_forms}\footnote{To be precise, we use the completely analogous flow version of Proposition~\ref{prop_normal_forms}.} now implies that the conjugacy is
$C^\infty$ along the strong stable and unstable leaves. It is smooth
along flow orbits, and the two applications of Journ\'e's lemma used
above give global $C^\infty$ regularity.
\end{proof}

\section{Proofs --- smooth conjugacy for Anosov symplectomorphisms}
Here we give proofs of smooth conjugacy results for Anosov symplectomorphisms assuming the results on symplecto-rigidity, which we will derive in the last section of this paper.

\begin{proof}[Proof of Theorem~\ref{thm_symp}]
Denote the conjugacy by $h$ ---  $h\circ f_1=f_2\circ h$. Then
$h^*\omega$ is a continuous $f_1$-invariant symplectic form. Since $h$
is homotopic to the identity, $h^*\omega$ belongs to the same cohomology class as
$\omega$. Hence, by symplecto-rigidity of $f_1$, $h^*\omega=\omega$.

Just as in the proof of Theorem~\ref{thm_contact}, on corresponding unstable leaves we can choose Kanai-parallel frames $\{e_1,\ldots,e_n\}$ and $\{\bar e_1,\ldots,\bar e_n\}$ such that
$$
Dh(e_i)=\bar e_i, \,\,\,i=1,\ldots n.
$$
The same construction applies on stable leaves. The exact same argument as in the proof of Theorem~\ref{thm_contact} then shows that $h$ is $C^{2+\eps}$.
\end{proof}

\begin{proof}[Proof of Corollary~\ref{cor_symp}]
By compactness, the $2$-pinching assumption implies $r$-pinching for
some $r<2$. Hence, by the discussion in Section~2, the stable and
unstable distributions of $f_1$ and $f_2$ are $C^{1+\eps}$ for some
$\eps>0$. Applying Theorem~\ref{thm_symp}, we obtain that the conjugacy
$h$ is $C^{2+\eps}$.

Let $\cH^i_x\colon E^u_i(x)\to W^u_i(x)$, $i=1,2$, be the
non-stationary linearizations from Proposition~\ref{prop_normal_forms}.
Since $h$ is $C^2$ and conjugates the dynamics, the maps
$$
h\circ \cH^1_x\circ (Dh_x|_{E^u_1(x)})^{-1}\colon
E^u_2(hx)\to W^u_2(hx)
$$
form a $C^2$ non-stationary linearization for $f_2$ satisfying the
normalization properties $(2)$--$(4)$. By uniqueness in
Proposition~\ref{prop_normal_forms}, this linearization coincides with
$\cH^2_{hx}$. Thus
$$
h|_{W^u_1(x)}=\cH^2_{hx}\circ Dh_x|_{E^u_1(x)}\circ(\cH^1_x)^{-1}.
$$
Therefore $h$ is $C^\infty$ along unstable leaves. The same argument for
stable leaves gives $C^\infty$ regularity along stable leaves. Journ\'e's
lemma now implies that $h$ is $C^\infty$.

\end{proof}

\begin{proof}[Proof of Corollary~\ref{cor_periodic_data}]
After conjugating by the toral automorphism induced by $h$, we may assume that $h$ is homotopic to the identity (and replace $\omega$ by the corresponding constant symplectic form). By~\cite[Proposition~1.6]{DWG}, $h$ is $C^{1+\nu}$ for some $\nu>0$. Since $h^*\omega$ is H\"older and $L$-invariant, symplecto-rigidity of $L$ gives $h^*\omega=\omega$. {Thus $h$ is symplectic, so symplecto-rigidity transfers from $L$ to $f$; also, the $C^1$ conjugacy transfers the $2$-pinching of $L$ to $f$.} Then Corollary~\ref{cor_symp} implies that $h$ is $C^\infty$.
\end{proof}

\begin{proof}[Proof of Corollary~\ref{cor_real_spectrum}]
The global version of Gogolev's rigidity theorem~\cite[Theorem~1.4]{DWG} (see also~\cite[Theorem~A]{Gog}) gives a conjugacy $h$, homotopic to the identity, of class $C^{1+\nu_0}$, where $\nu_0>0$ is uniform on a sufficiently small neighborhood $\U$ of $L$. The proof of Theorem~\ref{thm_dim4} allows us to shrink $\U$ further and choose an exponent $\alpha\in(0,\nu_0]$ such that every $f\in\U$ is $C^\alpha$-symplecto-rigid and $2$-pinched. Then $h^*\omega$ is $C^\alpha$, $f$-invariant, and, since $h$ is homotopic to the identity, cohomologous to $\omega$. Thus the $C^\alpha$-symplecto-rigidity of $f$ gives $h^*\omega=\omega$, so $h$ is symplectic. The $C^1$ conjugacy also transfers the $2$-pinching of $f$ to $g$. The bootstrap argument in the proof of Corollary~\ref{cor_symp} now applies verbatim and implies that $h$ is $C^\infty$.
\end{proof}

\begin{proof}[{Proof of Theorem~\ref{cor_product}}]
We can apply~\cite[Theorem~C]{Gog} to get a $C^{1+\nu}$ diffeomorphism $H$, homotopic to the identity, such that
$$
H\circ f\circ H^{-1}=L_\phi,
\qquad
L_\phi(x,y)=(Ax,By+\phi(x)).
$$
Let $\omega_0$ denote the standard product symplectic form and set $\omega_1=(H^{-1})^*\omega_0$. Then $\omega_1$ is a closed, non-degenerate, $L_\phi$-invariant $C^\nu$ form cohomologous to $\omega_0$.

Consider the product splitting $T_p\mathbb T^4=\mathbb R_x^2\oplus\mathbb R_y^2$, $p=(x,y)$. Since $DL_\phi|_{\mathbb R_y^2}=B$, we may write $(\omega_1)_p|_{\mathbb R_y^2}=a(p)\omega_B$, where $\omega_B$ is the standard area form on $\mathbb R_y^2$. By invariance of both forms, $a\circ L_\phi=a$, and transitivity of $L_\phi$ implies that $a$ is constant; integration over $\{x\}\times\mathbb T^2$ gives $a=1$. Thus $\mathbb R_y^2$ is a symplectic subspace. Its symplectic orthogonal is an invariant complement to $\mathbb R_y^2$, hence it is the graph of maps $P_p\colon\mathbb R_x^2\to\mathbb R_y^2$. Invariance gives
\begin{equation}
P_{L_\phi(p)}A=d\phi_x+BP_p.
\label{eq_graph_transform}
\end{equation}

\begin{lemma}
\label{lem:product_graph_base}
The map $P_p=P_{(x,y)}$ depends only on $x$. Moreover, $\omega_1$
is invariant under translations in the $y$-variable.
\end{lemma}
\begin{proof}
Fix $x\in\mathbb T^2$ and $y,y'\in\mathbb T^2$. Iterating the
difference of the two graph-transform identities gives, for every
$n\in\mathbb Z$,
$$
P_{L_\phi^n(x,y)}-P_{L_\phi^n(x,y')}
=B^n\bigl(P_{(x,y)}-P_{(x,y')}\bigr)A^{-n}.
$$
The left-hand side is bounded. In eigenbases of $A$ and $B$, each
coefficient on the right-hand side is multiplied by
$(\mu^{\pm1}/\lambda^{\pm1})^n$. The pinching inequalities imply that
none of these four ratios has absolute value $1$. Boundedness for both
positive and negative $n$ therefore forces every coefficient of
$P_{(x,y)}-P_{(x,y')}$ to vanish. Thus $P_{(x,y)}=P_{(x,y')}$. 

It remains to prove the assertion about $\omega_1$. Write
$P_{(x,y)}=P_x$, and let $\omega_A$ be the standard area form on
$\mathbb R_x^2$. Since the graph of $P_x$ is the $\omega_1$-orthogonal
complement of $\mathbb R_y^2$, there is a function $b$ such that
$$
\omega_{1,(x,y)}(u+P_xu,w+P_xw)=b(x,y)\omega_A(u,w).
$$
The graph-transform identity gives
$$
DL_\phi(u+P_xu)=Au+P_{Ax}Au.
$$
Since $\omega_1$ is $L_\phi$-invariant and $A$ preserves $\omega_A$,
it follows that $b\circ L_\phi=b$. Transitivity of $L_\phi$ therefore
implies that $b$ is constant. Finally, decomposing
$$
(u,v)=(u,P_xu)+(0,v-P_xu)
$$
and using that the two summands belong to $\omega_1$-orthogonal
subspaces, we obtain
$$
\omega_{1,(x,y)}\bigl((u,v),(w,z)\bigr)
=b\omega_A(u,w)+\omega_B(v-P_xu,z-P_xw).
$$
The right-hand side is independent of $y$, which proves that
$\omega_1$ is invariant under translations in the $y$-variable.
\end{proof}

By Lemma~\ref{lem:product_graph_base}, we may write $P_p=P_x$. The family of maps
$P_x\colon T_x\mathbb T^2\to\mathbb R_y^2$ can be considered as a $C^\nu$ one-form $P$
on $\T^2$ with values in $\mathbb R^2$. The mixed component
of $\omega_1$ is determined by
\begin{equation}
\label{eq:omega}
\omega_1(u,v)=-\omega_B(P_xu,v),
\qquad u\in\mathbb R_x^2,\quad v\in\mathbb R_y^2.
\end{equation}

\begin{lemma}
\label{lem:product_graph_exact}
There exists a $C^{1+\nu}$ map
$\Psi=(\psi_1,\psi_2)\colon\mathbb T_x^2\to\mathbb R^2$ such that
$P=d\Psi$.
\end{lemma}
\begin{proof}
We write $P=(p_1,p_2)$, where $p_1$ and $p_2$ are 1-forms.
We first show that $p_1$ is closed in the weak sense. Let $X$ and $Y$
be constant vector fields on $\mathbb T^2$. By~\eqref{eq:omega}
$\omega_1(X,\partial_{y_2})=-p_1(X)$, the identity $d\omega_1=0$
gives, in distributional sense
$$
\begin{aligned}
0=d\omega_1(X,Y,\partial_{y_2})
&=X\bigl[\omega_1(Y,\partial_{y_2})\bigr]
-Y\bigl[\omega_1(X,\partial_{y_2})\bigr]
+\partial_{y_2}\bigl[\omega_1(X,Y)\bigr]\\
&=-X[p_1(Y)]+Y[p_1(X)]
+\partial_{y_2}\bigl[\omega_1(X,Y)\bigr]\\
&=-dp_1(X,Y)+\partial_{y_2}\bigl[\omega_1(X,Y)\bigr].
\end{aligned}
$$
By Lemma~\ref{lem:product_graph_base}, $\omega_1$ is independent of
$y$, so the last term vanishes as a distribution. Hence
$dp_1(X,Y)=0$. Since $X$ and $Y$ were arbitrary, $dp_1=0$ in the weak
sense.

We next check that all periods of $p_1$ vanish. Let $\gamma$ be a
closed loop in $\mathbb T_x^2$ and let
$y_2(t)=(0,t)\pmod{\mathbb Z^2}$ be the $y_2$-coordinate loop. Since
$[\omega_1]=[\omega_0]$ and $\omega_0$ has no mixed component,
$$
0=\int_{\gamma\times y_2}(\omega_1-\omega_0)
=-\int_\gamma p_1.
$$
Thus $p_1$ has zero periods. Fix $x_0\in\mathbb T^2$. For
$x\in\mathbb T^2$,
choose a path $\sigma$ from $x_0$ to $x$ and set
$$
\psi_1(x)=\int_\sigma p_1.
$$
Because $p_1$ is closed and has zero periods, this definition is
independent of the choice of $\sigma$. Moreover,
$d\psi_1=p_1$, and the $C^\nu$ regularity of $p_1$ gives
$\psi_1\in C^{1+\nu}$. The argument for $p_2$ is analogous: it is weakly closed, and using
$\omega_1(X,\partial_{y_1})=p_2(X)$ together with the $y_1$-coordinate loop,
we obtain a $C^{1+\nu}$ function $\psi_2$ with $d\psi_2=p_2$. Therefore
$\Psi=(\psi_1,\psi_2)$ satisfies $P=d\Psi$.
\end{proof}

By Lemma~\ref{lem:product_graph_exact}, $P=d\Psi$. Substituting
in~\eqref{eq_graph_transform} yields
$$
d\Psi_{Ax}\circ A=B\circ d\Psi_x+d\phi_x.
$$
Here the chain rule gives
$$
d(\Psi\circ A)_x=d\Psi_{Ax}\circ A,
\qquad
d(B\circ\Psi)_x=B\circ d\Psi_x,
$$
because $A$ and $B$ are fixed linear maps. Thus
$$
d(\Psi\circ A-B\circ\Psi-\phi)=0,
$$
so there is a constant $c\in\mathbb R^2$ such that
$$
\Psi(Ax)-B\Psi(x)-\phi(x)=c.
$$
Since $1$ is not an eigenvalue of $B$, the map $I-B$ is invertible.
Replacing $\Psi$ by $\Psi-(I-B)^{-1}c$ therefore gives
$$
\Psi(Ax)=B\Psi(x)+\phi(x).
$$
Consequently, $S(x,y)=(x,y+\Psi(x))$ satisfies
$S\circ L=L_\phi\circ S$. Thus $h=S^{-1}\circ H$ is a $C^{1+\nu}$
diffeomorphism, homotopic to the identity, which conjugates $f$ to $L$.

After shrinking the neighborhood of $L$, the map $f$ is $2$-pinched and is symplecto-rigid by Addendum~\ref{add_dim4}; the automorphism $L$ is symplecto-rigid by Lemma~\ref{lemma_toral_forms}. Corollary~\ref{cor_symp} now implies that $h$ is $C^\infty$.
\end{proof}

\section{Proofs --- symplecto-rigidity}

We prove Theorem~\ref{thm_dim4}.

Let $f\colon(\T^4,\omega)\to (\T^4,\omega)$ be a symplectic perturbation of a symplectic automorphism $L\colon(\T^4,\omega)\to (\T^4,\omega)$. First note that $\omega$ has constant coefficients. Indeed, let $\omega_0$ be the unique constant-coefficient representative of the cohomology class $[\omega]$. Since $L^*[\omega]=[\omega]$, the constant-coefficient forms $L^*\omega_0$ and $\omega_0$ are cohomologous and hence equal. Thus $\omega-\omega_0$ is an invariant exact $2$-form, so Lemma~\ref{lemma_toral_forms} gives $\omega=\omega_0$.

Let $\omega'$ be an invariant $C^\alpha$ regular symplectic form in the same cohomology class as $\omega$. Then $\eta=\omega-\omega'=d\alpha$ is an exact $C^\alpha$ invariant 2-form. Then $\omega\wedge\eta$ is a 4-form which must be a constant multiple of $vol$ by ergodicity. On the other hand
$$
\int_M \omega\wedge\eta=\int_M d(\omega\wedge\alpha)=0.
$$
Hence $\omega\wedge\eta=0$, so $\eta$ is so called {\it primitive form}. Similarly we deduce that $\eta\wedge\eta=0$.

Since $\omega$ is non-degenerate we can represent $\eta$ via an endomorphism $B\colon TM\to TM$:
$$
\eta(v,u)=\omega(Bv,u)
$$
Invariance of $\eta$ and $\omega$ immediately implies that $B$ commutes with dynamics --- $B\circ Df=Df\circ B$, which in turn implies invariance of the stable and unstable subbundles --- $B(E^s)\subset E^s$, $B(E^u)\subset E^u$ and we can split $B=B^s\oplus B^u\colon E^s\oplus E^u\to E^s\oplus E^u$.

By invariance and uniform contraction, both $\eta$ and $\omega$ vanish on $E^s\wedge E^s$ and on $E^u\wedge E^u$. At any point $p$, choose bases $\{e_1,e_2\}$ of $E^s(p)$ and $\{f_1,f_2\}$ of $E^u(p)$ that are dual with respect to $\omega$. Then $\omega$ takes the standard form
$$
\omega=
e^1\wedge f^1+e^2\wedge f^2
.$$
At $p$, we can write $\eta$ as
$$
\eta=a e^1\wedge f^1 +b e^2\wedge f^1+ c e^1\wedge f^2+d e^2\wedge f^2
,$$
where $B^s_p=\begin{psmallmatrix}
    a & b\\
    c & d
\end{psmallmatrix}$ relative to $\{e_1,e_2\}$. The skew-symmetry of $\eta$ and $\omega$ gives $B^u_p=(B^s_p)^t$ relative to $\{f_1,f_2\}$. Finally, $\omega\wedge\eta=0$ yields $\operatorname{tr}(B_p^s)=a+d=0$, while $\eta\wedge\eta=0$ yields $\det(B_p^s)=ad-bc=0$. Hence $B^s_p$ is nilpotent.

Now given a periodic point $p=f^n(p)$ we have that $B^s_p$ commutes with $D_pf^n|_{E^s(p)}$ which immediately implies that either $D_pf^n|_{E^s(p)}$ has equal real eigenvalues or $B^s_p$ vanishes.

\begin{lemma}\label{lemma_real} If $L$ has real stable eigenvalues of distinct moduli then $\eta\equiv 0$.
\end{lemma}
\begin{proof}
Let $\lambda_1,\lambda_2$ be the stable eigenvalues of $L$ and assume
$|\lambda_1|<|\lambda_2|<1$. Then the splitting of $E^s_L$ into the
corresponding one-dimensional eigenspaces is dominated. Hence, for any
$f$ sufficiently $C^1$ close to $L$, we have a dominated splitting
$$
E^s=E^s_1\oplus E^s_2
$$
with one-dimensional subbundles.

By the preceding observation, at every periodic point $p=f^n(p)$ the
nilpotent operator $B^s_p$ commutes with a linear map with two distinct
real eigenvalues. Therefore $B^s_p=0$. Periodic points are dense and
$B^s$ is continuous, hence $B^s\equiv0$. Since $B^u=(B^s)^t$ in the
dual bases above, we also have $B^u\equiv0$. Thus $B\equiv0$ and
$\eta\equiv0$.
\end{proof}

Let $\ell^s\subset E^s$ be the kernel of $B^s$ and, similarly, let
$\ell^u\subset E^u$ be the kernel of $B^u$. They are invariant
distributions of dimension $\ge 1$, note that $\dim \ell^s(x)=\dim
\ell^u(x)$, $\forall x\in \T^4$. The set $K$ where this distribution
coincides with $E^s$ is closed and invariant, this is precisely the set
where $\eta$ vanishes. If $K=\T^4$ then $\eta\equiv 0$ and we are done.
Note that if $K\neq \T^4$ then $K$ has empty interior because it is
invariant. In particular, we already have that all periodic points in
$\T^4\backslash K$ have repeated real eigenvalues.

\begin{lemma} If $K\neq \T^4$ then $K=\varnothing$.
\end{lemma}
\begin{proof}
We show that $K$ is saturated by stable and unstable leaves. Recall that $\eta$ is $C^\alpha$. Since the distinct-real-moduli case has already been excluded, all stable eigenvalues of $L$ have a common modulus $\rho<1$ and, by symplecticity, all unstable eigenvalues have modulus $\rho^{-1}$. Possible Jordan blocks contribute only polynomial factors, which can be absorbed into $\kappa^n$ for any $\kappa>1$. Fix $\kappa>1$ so close to $1$ that $\rho^\alpha\kappa^{\alpha+2}<1$. Then, after shrinking the $C^1$ neighborhood of $L$ if needed, we have
$$
d(f^n x,f^n y)\le C(\rho\kappa)^n d(x,y),\qquad
\|Df^n|_{E^s(y)}\|\le C(\rho\kappa)^n,
\qquad
\|Df^n|_{E^u(y)}\|\le C(\rho^{-1}\kappa)^n.
$$
Setting $\theta:=\rho^\alpha\kappa^{\alpha+2}<1$, we obtain
$$
d(f^n x,f^n y)^\alpha
\|Df^n|_{E^s(y)}\|\,\|Df^n|_{E^u(y)}\|\le C\theta^n
$$
for all $y\in W^s_{\rm loc}(x)$ and all $n\ge0$, and the analogous estimate for $f^{-1}$.

Let $x\in K$ and $y\in W^s_{\rm loc}(x)$. Given unit vectors
$v,w\in T_y\T^4$, write $v=v^s+v^u$, $w=w^s+w^u$. Since $\eta$ vanishes
on $E^s\wedge E^s$ and $E^u\wedge E^u$, invariance gives
$$
\eta_y(v,w)=
\eta_{f^n y}(Df^n v^s,Df^n w^u)
+\eta_{f^n y}(Df^n v^u,Df^n w^s).
$$
Since $f^n x\in K$, we may subtract $\eta_{f^n x}$, using the standard
trivialization of $T\T^4$. By the H\"older regularity of $\eta$ and the
above estimate, the right-hand side is bounded by $C\theta^n$. Letting
$n\to\infty$ gives $\eta_y(v,w)=0$.
Hence $W^s_{\rm loc}(x)\subset K$. Applying the same argument to
$f^{-1}$ gives $W^u_{\rm loc}(x)\subset K$.

Now, if $K$ is non-empty, then local product structure implies $K=\T^4$, which is a contradiction.
\end{proof}

We now finish the proof. Assume that $\eta\not\equiv0$. Then
$K\neq\T^4$ and, by the lemma, $K=\varnothing$. Hence $B^s$ and $B^u$
have rank one at every point, and $\ell^s=\ker B^s$ is an invariant line
field. We still have two cases. If $L$ has complex stable eigenvalues,
then, for $f$ sufficiently $C^1$ close to $L$, the stable derivative at
the continuation of a fixed point of $L$ also has complex eigenvalues.
This is incompatible with the invariant line $\ell^s$. Thus $L$ has real
stable eigenvalues. By Lemma~\ref{lemma_real}, they cannot have distinct
moduli; passing to the second iterate if necessary, we may assume that
they are equal, and we proceed to argue in this case. Since $B^s$ is nilpotent,
$\operatorname{Im}B^s=\ell^s$. Also, from the formula for $\eta$ above,
$\ell^u$ is the $\omega$-annihilator of $\ell^s$ inside $E^u$. Therefore
$$
\ker\eta=\ell^s\oplus \ell^u.
$$

Subbundles $E^s_L$ and $E^u_L$ are Lagrangian subspaces and $\omega$ pairs them
non-degenerately. Choose a linear map $J\colon E^s_L\to E^u_L$ such
that
$$
\omega(v,Jv)>0,\qquad v\neq0.
$$
Set
$$
P_0=\{v+Jv:v\in E^s_L\}.
$$
For a line $\ell\subset E^s_L$ denote by $\ell^\perp$ its
$\omega$-annihilator in $E^u_L$ and set $H_\ell=\ell\oplus \ell^\perp$.
Then $P_0$ is transverse to every $H_\ell$. Indeed, if
$v+Jv\in H_\ell$, then $v\in\ell$ and $Jv\in\ell^\perp$, so
$\omega(v,Jv)=0$. This contradicts $\omega(v,Jv)>0$ unless $v=0$.

By compactness this transversality is uniform in $\ell$. Since $f$ is
$C^1$ close to $L$, the planes $\ker\eta_x=\ell^s(x)\oplus \ell^u(x)$ are
uniformly close to the family $\{H_\ell\}_{\ell\subset E^s_L}$. Hence
$P_0$ is transverse to $\ker\eta_x$ for all $x\in\T^4$. Rational
two-planes are dense in the Grassmannian, so we can choose a rational
two-plane $P$ close to $P_0$ such that we still have
$$
P\cap\ker\eta_x=\{0\},\qquad x\in\T^4.
$$

Regard $P$ as a translation-invariant plane field on $\T^4$, and let
$T_P=P/(P\cap\Z^4)$ be the corresponding rational torus. The
restriction $\eta|_{T_P}$ is a continuous nowhere vanishing $2$-form.
After orienting $T_P$ we can write $\eta|_{T_P}=\rho\,dA$, where
$\rho$ is continuous and nowhere zero. Hence $\rho$ has constant sign
and
$$
\int_{T_P}\eta\neq0.
$$
This contradicts exactness of $\eta$, since an exact form pairs
trivially with every closed cycle. Thus $\eta\equiv0$.

\begin{proof}[Proof of Addendum~\ref{add_dim4}]
Let $\omega'$ be an invariant continuous symplectic form in the same
cohomology class as $\omega$ and set $\eta=\omega-\omega'$. We need to
show that $\eta\equiv 0$. 

The case of distinct real stable eigenvalues was already treated in
Lemma~\ref{lemma_real}; its proof only uses continuity of $B$. Thus we
are in the remaining case. Assume that $\eta\neq0$. As in the proof of
Theorem~\ref{thm_dim4}, let $K=\{x:\eta_x=0\}$. By ergodicity, $K$ has
zero measure. On the full measure set $\T^4\backslash K$ we have
$$
\ker\eta_x=\ell^s(x)\oplus \ell^u(x).
$$
The line fields $\ell^s$ and $\ell^u$ are invariant and are measurable when viewed as line fields on $\T^4$.

Assume for a moment that $f$ has two distinct stable Lyapunov exponents
$\chi_1<\chi_2<0$, and let
$$
E^s=E^s_1\oplus E^s_2
$$
be the stable Oseledets splitting for $vol$. The equivariance and boundedness
of $B^s$ imply that it preserves $E^s_1$. Indeed, if $v\in E^s_1(x)$, then
$$
\|Df_x^n(B^s_xv)\|
=\|B^s_{f^n x}(Df_x^n v)\|
\le\|B^s\|_{C^0}\|Df_x^n v\|,
$$
so any non-zero $B^s_xv$ has Lyapunov exponent $\chi_1$ and therefore
belongs to $E^s_1(x)$. Since $B^s_x$ is nilpotent, its restriction to the
one-dimensional space $E^s_1(x)$ vanishes. Thus $\ell^s=E^s_1$ almost
everywhere. Choose measurable unit vectors $e_i(y)\in E^s_i(y)$. For almost
every $y\notin K$, the operator $B^s_y$ is non-zero and nilpotent, with
$\operatorname{Im}B^s_y=\ker B^s_y=\ell^s(y)$. Thus, in the basis
$\{e_1(y),e_2(y)\}$,
$$
B^s_y=\begin{pmatrix}0&b(y)\\0&0\end{pmatrix},
\qquad b(y)\neq0.
$$
Choose $\delta>0$ such that the set
$\Omega_\delta=\{y:|b(y)|\ge\delta\}$ has positive measure, and choose an
Oseledets-regular point $x\in\Omega_\delta$ which is recurrent to
$\Omega_\delta$. In the Oseledets bases at $x$ and $f^n x$, we have
$$
Df_x^n|_{E^s(x)}=\begin{pmatrix}a_n&0\\0&d_n\end{pmatrix}.
$$
The relation
$$
B^s_{f^n x}\circ Df_x^n|_{E^s(x)}
=Df_x^n|_{E^s(x)}\circ B^s_x
$$
gives
$$
b(f^n x)d_n=a_n b(x).
$$
Along return times $n_j$ for which $f^{n_j}x\in\Omega_\delta$, the ratio
$|a_{n_j}/d_{n_j}|$ is bounded above and below by positive constants
independent of $j$. Hence,
$$
\chi_1-\chi_2=\lim_{j\to\infty}\frac{1}{n_j}
\log\left|\frac{a_{n_j}}{d_{n_j}}\right|
=0,
$$
a contradiction.

 Therefore $vol$ has only one stable
Lyapunov exponent and, by symplecticity, only one unstable Lyapunov exponent.
Also, since $f$ is $C^1$ close to $L$, the stable and
unstable derivative cocycles are fiber bunched.
Thus Kalinin--Sadovskaya~\cite[Theorem~3.3 and
Corollary~3.8]{KS13} applies and implies that $\ell^s$ and $\ell^u$ coincide almost
everywhere with H\"older invariant line fields on $\T^4$. We keep the
same notation for these extensions; then
$\ker\eta\supset \ell^s\oplus \ell^u$.

Fix a smooth Riemannian metric and let
$P=(\ell^s\oplus \ell^u)^\perp$. Then $P$ is a H\"older two-dimensional bundle. Passing to
an iterate and, if necessary, to a double cover,
we may assume that the relevant bundles are oriented and that their orientations are preserved; this does not
affect the conclusion. Write $Df$ with respect to
$(\ell^s\oplus \ell^u)\oplus P$ as
$$
Df_x=
\begin{pmatrix}
A_x&C_x\\
0&Q_x
\end{pmatrix},
\qquad Q_x\colon P_x\to P_{fx}.
$$
The cocycle $Q_x$ is H\"older. Let $\rho_x$ be the area form on $P_x$
induced by the Riemannian metric, and define the positive inverse Jacobian $J_\rho$ by
$$
Q_x^*\rho_{fx}=J_\rho(x)^{-1}\rho_x .
$$
Then $J_\rho$ is a positive H\"older function. Since
$\ell^s\oplus \ell^u\subset\ker\eta$, we can write
$$
\eta|_{P_x}=u(x)\rho_x
$$
for a continuous function $u\colon\T^4\to\mathbb R$. Invariance of $\eta$ gives, for
$v,w\in P_x$,
$$
u(x)\rho_x(v,w)
=u(fx)\rho_{fx}(Q_xv,Q_xw)
=u(fx)J_\rho(x)^{-1}\rho_x(v,w).
$$
Hence
$$
u(x)J_\rho(x)=u(fx).
$$
The zero set of $u$ is invariant. Thus either $u=0$ a.e., in which case
$\eta\equiv0$ by continuity, or $u$ is non-zero a.e. In the latter case the
sign of $u$ is invariant and hence constant a.e. Reversing the orientation
of $P$ if necessary, we may assume that $u>0$ a.e. Applying the measurable
Livshits theorem~\cite{Liv72} to
$$
\log J_\rho=\log (u\circ f)-\log u
$$
we conclude that $\log u$ coincides a.e. with a H\"older function
$\varphi$. Hence $u=e^\varphi$ a.e. Since both sides are continuous, this
equality holds everywhere. Thus $u$ is positive and H\"older, so its
vanishing set $K$ is empty. The final argument in the preceding proof of
Theorem~\ref{thm_dim4} now applies and gives $\eta\equiv0$.
\end{proof}

Now we prove Theorem~\ref{theorem_residual}. We will use the same idea.
Write $\dim M=2d$. Let $\{U_j\}$ be a countable basis of open sets in
$M$. Denote by $\mathcal G_j$ the set of Anosov symplectomorphisms which
have a periodic point $p=f^n(p)$ in $U_j$ such that
$D_pf^n|_{E^s(p)}$ has simple spectrum over $\mathbb C$. The set $\mathcal G_j$ is open:
the periodic point has a continuation and simplicity of the spectrum over $\mathbb C$ persists.
Also $\mathcal G_j$ is dense. Indeed, periodic points are dense for an
Anosov diffeomorphism, and the symplectic Franks' lemma~\cite[Lemma~2.3]{CH17} allows us to
perturb the derivative along the orbit so that the stable return has
simple spectrum over $\mathbb C$. Hence
$$
\mathcal G=\bigcap_j \mathcal G_j
$$
is residual. For every $f\in\mathcal G$, periodic points with simple
stable spectrum over $\mathbb C$ are dense.

\begin{lemma}
\label{lem:wedgevanishing}
Let $\eta$ be an invariant exact continuous $2$-form. Then
$$
\eta^k\wedge\omega^{d-k}=0,\qquad k=1,\ldots,d.
$$
\end{lemma}
\begin{proof}
Each form $\eta^k\wedge\omega^{d-k}$ is an invariant exact top-degree
form. Hence, by ergodicity, it is a constant multiple of $\omega^d$.
Since it is exact, its integral is zero. Therefore this multiple is
zero.
\end{proof}

\begin{lemma}
\label{lem:nilpotent}
Let $B^s\colon E^s\to E^s$ be defined by
$$
\eta(v^s,v^u)=\omega(B^s v^s,v^u).
$$
Then $B^s_x$ is nilpotent for every $x$.
\end{lemma}
\begin{proof}
As before, $\eta$ vanishes on $E^s\wedge E^s$ and on
$E^u\wedge E^u$, and we have the splitting $B=B^s\oplus B^u$. Choose
dual bases of $E^s(x)$ and $E^u(x)$ with respect to $\omega$. Then
$$
(\omega+t\eta)^d=\det(I+tB^s_x)\,\omega^d.
$$
By Lemma~\ref{lem:wedgevanishing}, all non-constant coefficients in
this polynomial vanish. Hence $\det(I+tB^s_x)\equiv1$. Therefore all
eigenvalues of $B^s_x$ are equal to zero; hence $B^s_x$
is nilpotent.
\end{proof}

Let $f\in\mathcal G$ and let $\omega'$ be an invariant continuous
symplectic form in the cohomology class of $\omega$. Set
$\eta=\omega'-\omega$. Then $\eta$ is invariant and exact. As before,
$\eta$ vanishes on $E^s\wedge E^s$ and on $E^u\wedge E^u$, and
$$
\eta(v^s,v^u)=\omega(B^s v^s,v^u).
$$
Invariance gives
$$
B^s_{f x}\circ Df|_{E^s(x)}=Df|_{E^s(x)}\circ B^s_x.
$$
Thus, if $p=f^n(p)$, then $B^s_p$ commutes with
$D_pf^n|_{E^s(p)}$. By Lemma~\ref{lem:nilpotent}, $B^s_p$ is
nilpotent. Since $D_pf^n|_{E^s(p)}$ has simple spectrum over $\mathbb C$, this implies
$B^s_p=0$. Since such periodic points are dense, $B^s\equiv0$. Then
also $B^u\equiv0$, and therefore $\eta\equiv0$. This proves that every
$f\in\mathcal G$ is symplecto-rigid.


\appendix

\section{Invariant forms for toral automorphisms}
\label{app_toral_forms}

Here we give a continuous version of Fang's result in the case of toral
automorphisms.

\begin{lemma}
\label{lemma_toral_forms}
Let $L\colon\T^d\to\T^d$ be a hyperbolic toral automorphism. Then any
continuous invariant exact $2$-form on $\T^d$ vanishes.
\end{lemma}
\begin{proof}
We use the same symbol $L\in GL(d,\Z)$ for the matrix defining the
automorphism. Let $\eta$ be a continuous invariant exact $2$-form.
Denote its 2-form Fourier coefficients by $\widehat\eta(k)$.
Then $\eta$ has the Fourier expansion
$$
\eta_x=\sum_{k\in\Z^d}e^{2\pi i\langle k,x\rangle}\widehat\eta(k),
$$
where the equality is understood in $L^2$.
Invariance $L^*\eta=\eta$ gives
$$
\widehat\eta((L^t)^n k)=(L^*)^n\widehat\eta(k),
\qquad n\in\Z.
$$
If $k\neq0$, then $|(L^t)^n k|\to\infty$ both as $n\to\infty$ and as
$n\to-\infty$. By the
Riemann--Lebesgue lemma we have
$$
(L^*)^n\widehat\eta(k)\longrightarrow0
$$
both as $n\to\infty$ and as $n\to-\infty$.  Hence
$\widehat\eta(k)=0$ for every $k\neq0$.

Thus all non-constant Fourier coefficients of $\eta$ vanish. By uniqueness
of Fourier coefficients for continuous functions, $\eta$ is a constant
form. But $\eta$ is exact; hence $\eta\equiv0$.
\end{proof}

\section{Regularity property of the holonomy}

Here we sketch the proof of Lemma~\ref{lemma}. Namely, we show that the vector field $W(\gamma(t))$ over a curve $\gamma\subset W^u(a)$, $a=\gamma(0)$, given by unstable holonomy, $W(\gamma(t))=DH^u_{a\to\gamma(t)}(W_0)$, is a $C^r$ vector field provided that the stable and unstable distributions are $C^r$, $r\ge1$.

For simplicity, assume that we are in dimension two and $a=(0,0)$. Let $\gamma(x)=(x,0)$ be the curve in the unstable leaf given by the $x$-axis, so that $F(x,0)=0$, and let $W^s(a)=\{(0,y):y\in\R\}$. The stable and unstable distributions are
$$
E^s(x,y)=\operatorname{span}\left\{ g(x,y)\frac{\partial}{\partial x}+\frac{\partial}{\partial y}\right\}
$$
and
$$
E^u(x,y)=\operatorname{span}\left\{\frac{\partial}{\partial x}+F(x,y)\frac{\partial}{\partial y}\right\},
$$
where $g$ and $F$ are $C^r$, and hence at least $C^1$, by assumption.

The unstable manifolds are given by the solutions of the ODE
\begin{multline*}
\alpha_x(x,y)=F(x,\alpha(x,y));\\
\shoveleft{\alpha(0,y)=y. \hfill}
\end{multline*}
By standard ODE theory, the solution $\alpha$ is unique and is $C^1$ as a function of two variables. Linearizing the ODE along the $x$-axis gives the vector field $\alpha_y(x,0)\frac{\partial}{\partial y}$ and, accordingly, the vector field $W$ is given by
$$
W(x,0)=\alpha_y(x,0)\left(\frac{\partial}{\partial y}+g(x,0)\frac{\partial}{\partial x}\right)
.$$
Hence, we need to show that $\alpha_y(x,0)$ is a $C^1$ function of $x$; in fact, we prove that $\alpha_y(x,y)$ is $C^1$ in $x$ for every $y$. Consider the approximations
$$
\alpha(x,y;\Delta y)=\big(\alpha(x,y+\Delta y)-\alpha(x,y)\big)/\Delta y
.$$
Differentiating these approximations with respect to $x$ and using the ODE gives
$$
\alpha_x(x,y;\Delta y)=\big(F(x,\alpha(x,y+\Delta y))-F(x,\alpha(x,y))\big)/\Delta y,
$$
which converges uniformly on compact $x$-intervals as $\Delta y\to 0$ to
$F_y(x,\alpha(x,y))\alpha_y(x,y)$. On the other hand, since $\alpha$
is $C^1$, the functions $\alpha(x,y;\Delta y)$ converge uniformly to
$\alpha_y(x,y)$. The fundamental theorem of calculus gives
$$
\alpha(x,y;\Delta y)
=\alpha(0,y;\Delta y)
+\int_0^x\alpha_x(s,y;\Delta y)\,ds
=1+\int_0^x\alpha_x(s,y;\Delta y)\,ds,
$$
where we used the initial condition $\alpha(0,y)=y$. Moreover,
$$
\left|\int_0^x\alpha_x(s,y;\Delta y)\,ds
-\int_0^xF_y(s,\alpha(s,y))\alpha_y(s,y)\,ds\right|
\le x\sup_{s}\left|\alpha_x(s,y;\Delta y)
-F_y(s,\alpha(s,y))\alpha_y(s,y)\right|\longrightarrow 0.
$$
Thus we may pass to the limit in the preceding integral identity and
obtain
$$
\alpha_y(x,y)=1+\int_0^x
F_y(s,\alpha(s,y))\alpha_y(s,y)\,ds.
$$
Consequently,
$$
\frac{\partial}{\partial x}\alpha_y(x,y)
=F_y(x,\alpha(x,y))\alpha_y(x,y),
$$
and $\alpha_y(x,y)$ is $C^1$ in $x$. In particular,
$\alpha_y(x,0)$ is $C^1$ in $x$, and the displayed formula for $W$
shows that $W$ is $C^1$. This finishes the proof in the case when $r=1$. If $r=1+\theta$,
$0<\theta<1$, then the right-hand side is $\theta$-H\"older in $x$,
and hence both $\alpha_y(x,0)$ and $W$ are $C^{1+\theta}$ in $x$. More generally, differentiating the variational equation inductively shows that $\alpha_y(x,0)$ and $W$ are $C^r$ in $x$ for every $r\ge1$.

Finally, the higher-dimensional case can be proved using the same method, but the notation for the derivatives is cumbersome.
\bibliographystyle{alpha}

\bibliography{bib}

\end{document}